\documentclass[final,leqno]{siamltex}

\newcommand{\Frac}[2] {\frac {\textstyle #1} {\textstyle #2}}

\newcommand{\Int }    {\displaystyle \int}

\newcommand{\Sum }    {\displaystyle \sum}

\newcommand{\Lim} {\displaystyle \lim}

\def \ro {{\rho}}
\def\R {{ \mathbb{R} }}
\def\N {{ \mathbb{N} }}
\def \no {{\nonumber}}

\def \Theta {{\mathcal{U}}}

\def \epso {{\nu_2}}
\def \nuo {{\nu_1}}
\def \Ag {{A_{\gamma}}}
\def \fg {{f_{\gamma}}}
\def \Fga {{F_{\gamma}}}
\def \B {{ {\mathcal B} }}
\def\dist{{\rm dist}}

\def \F {{ { \mathcal F} }}
\def \M {{ {\Omega} }}

\def \Vo {{\mathbf{V}}}
\def \Yo {{\mathbf{Y} }}
\def \G {{ {\mathcal G} }}

\def \C {{ {\mathcal C} }}
\def \V {{ {\mathcal V} }}
\def \P {{ {\mathcal P} }}

\def \A {{ {\mathcal A} }}
\def \E {{ {\mathcal E} }}

\def \K {{ {\mathcal K} }}
\def \l {{ { _{L^2}} }}
\def \R {{ {\boldmath{R}} }}
\def\T {{{\boldmath{T}}}}
\def\N {\boldmath{N}}

\def \ut {{\tilde{{u}}_k(t)}}
\def \pt {{\tilde{{\phi}}_k(t)}}

\newcommand{\dt}{k}
\newcommand{\beq}{\begin{equation}}
\newcommand{\eeq}{\end{equation}}
\newcommand{\bsp}{\begin{split}}
\newcommand{\esp}{\end{split}}
\newcommand{\bg}{\begin{gathered}}
\newcommand{\eg}{\end{gathered}}

\newcommand{\bpf}{\begin{proof}}
\newcommand{\epf}{\end{proof}}

\newtheorem{Thm}{Theorem}
\newtheorem{Lem}{Lemma}
\newtheorem{Cor}{Corollary}[section]
\newtheorem{Rem}{Remark}[section]
\newtheorem{Prop}{Proposition}
\newtheorem{Def}{Definition}[section]

\newcommand{\gvi}{\|g\|_{\infty}}

\newcommand{\ph}{\phi^n}
\newcommand{\phn}{\phi^{n-1}}
\newcommand{\g}{_{\gamma}}
\newcommand{\x}{\xi}

\title{Long time stability of a classical efficient scheme for an incompressible two-phase flow model}

\author{T. Tachim Medjo \thanks{Department of
Mathematics and Statistics, Florida International University, DM413B,
University Park, Miami, Florida 33199, USA  ({\tt tachimt@fiu.edu}).}
        \and F. Tone\thanks{ Department Mathematics and Statistics, University of West
Florida, Pensacola, FL 32514, USA ({\tt ftone@uwf.edu}).}}

\begin{document}

\maketitle

\begin{abstract}
In this article we consider the implicit Euler scheme for  a
homogeneous two-phase flow model in a two-dimensional domain and
with the aid of the discrete Gronwall lemma and of the discrete
uniform Gronwall lemma we prove that the global attractors generated
by the numerical scheme converge to the global attractor of the
continuous system as the time-step approaches zero.
\end{abstract}



\pagestyle{myheadings}
\thispagestyle{plain}
\markboth{T. Tachim Medjo  and F. Tone }{2D two-phase flow}

\section{Introduction}
It is well known that the 2D incompressible flow can be extremely
complicated with possible chaos and turbulent behavior, \cite{temam88}.
Although some of the features of this turbulent or chaotic behavior
may be deduced via analytic means, it is widely believed that
numerical methods are indispensable for obtaining a better
understanding of these complicated phenomena.

Let us recall that the incompressible Navier--Stokes equations
govern the motion of single-phase fluids, such as air or water. On
the other hand, we are faced with the difficult problem of
understanding the motion of binary fluid mixtures, that is fluids
composed by either two phases of the same chemical species or phases
of different composition. Diffuse interface models are well-known
tools to describe the dynamics of complex (e.g., binary) fluids,
\cite{gra1}. For instance, this approach is used in \cite{bles} to
describe cavitation phenomena in a flowing liquid. The model
consists of the Navier--Stokes equations coupled with the
phase-field system, \cite{cagi,gra1,gra2,gra3}. In the isothermal
compressible case, the existence of a global weak solution is proved
in \cite{fei1}. In the incompressible isothermal case, neglecting
chemical reactions and other forces, the model reduces to an
evolution system which governs the fluid velocity $u $ and the order
parameter $\phi. $ This system can be written as a Navier--Stokes
equation coupled with a convective Allen-Cahn equation, \cite{gra1}.
The associated initial and boundary value problem was studied in
\cite{gra1}, in which the authors proved that the system generated a
strongly continuous semigroup on a suitable phase space which
possesses a global attractor $\A. $ They also established the
existence of an exponential attractor $\E. $ This entails that $\A $
has a finite fractal dimension, which is estimated in \cite{gra1} in
terms of some model parameters. The dynamic of simple single-phase
fluids has been widely investigated, although some important issues
remain unresolved, \cite{temam88}. In the case of binary fluids, the
analysis is even more complicated and the mathematical study is
still at its infancy, as noted in \cite{gra1}.

In this article, we consider a homogeneous two-phase flow model in a
two-dimensional domain, we discretize in time using the implicit
Euler scheme and with the aid of the discrete Gronwall lemma and of
the discrete uniform Gronwall lemma we prove that the global
attractors generated by the numerical scheme converge to the global
attractor of the continuous system as the time-step approaches zero.
Our work has been inspired by previous results of one of the authors
and her collaborators. In \cite{TW},
for example, the authors considered 
the implicit Euler scheme for the 2D Navier--Stokes equations and
proved that the numerical scheme was $H^1$-uniformly stable in time.
In a later article (see \cite{CZT}), the authors used the theory for
multi-valued attractors to prove the convergence of the discrete
attractors to the global attractor of the continuous system as the time-step parameter approached zero.
In \cite{T4}, the author considered the implicit Euler scheme for the
two-dimensional magnetohydrodynamics equations and showed that the
scheme was $H^2$-stable. Similar results were obtained in \cite{TW2011} and \cite{ET}, 
where the authors proved not only the long-time stability of the
implicit Euler scheme for the two-dimensional Rayleigh-Benard
convection problem, and the thermohydraulics equations,
respectively, but also the convergence of the global attractors
generated by the numerical scheme to the global attractor of the
continuous system as the time-step approaches zero.

Let us mention that although we drew our inspiration from
\cite{CZT,TW,T4,TW2011}, the problem we treat here does not fall
into the framework of these references. Besides the usual nonlinear
term of the conventional Navier--Stokes system, the model considered
here contains another (stronger) nonlinear term that results from
the coupling of the convective Allen-Cahn equation and the
Navier--Stokes system. Because of this, the analysis of the
numerical scheme considered in this work tends to be more
complicated and subtle than that of the 2D Navier--Stokes system
studied in \cite{TW}.

The article is divided as follows. In the next section, we recall
from \cite{gra1} the incompressible homogeneous two-phase flow and
its mathematical setting. In Section 3 we study the stability of a
time discretization scheme for the model. More precisely, we prove
that the scheme is uniformly bounded in $\Yo $ and $\Vo$, provided
that the time-step is small enough. In Section \ref{s5} we recall
the theory of the so-called multi-valued attractors, and then we
apply it to our model.

\section{A two phase flow model and its mathematical setting}
\subsection{Governing equations}
In this article, we consider a model of homogeneous incompressible
two-phase flow with singularly oscillating forces. More precisely,
we assume that the domain $\M $ of the fluid is a bounded domain in
$\R^2. $ Then, we consider the system

\begin{equation}\label{ye6}
\left \{
\begin{array}{lll}
\Frac{\partial u}{\partial t} -  \nuo \Delta u
+ (u \cdot \nabla) u + \nabla p  = g -  \K \hbox{div} (\nabla \phi \otimes \nabla \phi), \\ \\
\hbox{ div } u  = 0, \\ \\
 \Frac{\partial \phi }{\partial t } + u \cdot \nabla \phi + \mu = 0,
 \\ \\
 \mu = - \epso \Delta \phi + \alpha f(\phi),
\end{array} \right. \end{equation}
in $\M \times (0, + \infty). $

 In (\ref{ye6}), the unknown functions are the velocity $u = (u_1,
u_2) $ of the fluid, its pressure $p $ and the order (phase)
parameter $ \phi. $  The quantity $\mu $ is the variational
derivative of the following free energy functional
\begin{equation}
\F(\phi) = \Int_{\M} \left ( \frac{\epso}{2} | \nabla \phi |^2 +
\alpha F(\phi) \right ) ds,
\end{equation}
where, e.g., $ F(r) = \Int^r_0 f(\zeta) d \zeta. $ Here, the
constants $ \nuo > 0 $ and $\K > 0 $ correspond to the kinematic
viscosity of the fluid and the capillarity (stress) coefficient
respectively, $ \epso, \ \alpha > 0 $ are two physical parameters
describing the interaction between the two phases. In particular, $
\epso $ is related with the thickness of the interface separating
the two fluids. Hereafter, as in \cite{gra1}, we assume that $ \epso
\leq \alpha. $

In (\ref{ye6}), $g$ is an external time-dependent volume force and we have assumed the density equal to one.

We endow (\ref{ye6}) with the boundary condition
\begin{equation}\label{b2}
u = 0, \ \frac{\partial \phi}{\partial \eta} = 0 \hbox{ on }
\partial \M \times (0, + \infty),
\end{equation}
where $\partial \M $ is the boundary of $\M $ and $\eta $ is its
outward normal.

The initial condition is given by
\begin{equation}\label{ib2}
(u, \phi) (0)  = (u_0,\phi_0) \hbox{ in } \M.
\end{equation}

\subsection{Mathematical setting}

We first recall from \cite{gra1} the weak formulation of
(\ref{ye6})--(\ref{ib2}). Hereafter, we assume that the domain $\M$ is
bounded with a smooth boundary $\partial \M $ (e.g., of class $\C^2
).$ We also assume that $f \in \C^1(\R) $ satisfies
\begin{equation}\label{s1}
 \left \{
\begin{array}{ll}
\Lim_{|r| \rightarrow + \infty} f^{'}(r) > 0, \\ \\
| f^{'}(r) | \leq c_f ( 1 + |r|^m ), \ \forall r \in \R,
\end{array} \right. \end{equation}
where $c_f $ is some positive constant and $m \in [1,+ \infty) $ is
fixed. It follows from (\ref{s1}) that
\begin{equation}\label{s2}
|f(r)| \leq c_f ( 1 + |r|^{m+1} ), \ \forall r \in \R.
\end{equation}

If $ X $ is a real Hilbert space with inner product $(\cdot,
\cdot)_X, $ we will denote the induced norm by $| \cdot |_X, $ while
$X^* $ will indicate its dual. We set $$ \V = \{ u \in \C^{\infty}_c
(\M): \ \hbox{ div } u = 0 \hbox{ in } \M \}. $$ We denote by $H $
and $V $ the closure of $\V$ in $(L^2(\M))^2 $ and $(H^1_0(\M))^2 $
respectively.  The scalar product in $H $ is denoted by $(\cdot,
\cdot)_{L^2} $ and the associated norm by $ | \cdot |_{L^2}. $
Moreover, the space $V $ is endowed with the scalar product
$$ ((u, v)) = \Sum^2_{i=1} (\partial_{x_i} u, \partial_{x_i} v)_{L^2}, \ \
\Arrowvert u \Arrowvert = ((u, u))^{1/2},$$ and we have the
Poincar\'e inequality
\beq\label{Poin} |u|_{L^2}^2\leq c_\Omega
\|u\|^2, \quad \forall u \in V.
\eeq

We now define the operator $A $ by
$$ A u = -\P \Delta u, \ \forall u \in D(A) = H^2(\M) \cap V, $$
where $\P $ is the Leray-Helmotz projector of $L^2(\M) $ onto $H. $
Then $A $ is a self-adjoint positive unbounded operator in $H $
which is associated with the scalar product defined above.
Furthermore, $ A^{-1} $ is a compact linear operator on $H $ and $ |
A \cdot |_{L^2} $ is a norm on $D(A) $, equivalent to the
$H^2$-norm.

Note that from (\ref{s1}), we can find $\gamma > 0 $ such that
\begin{equation}\label{s3}
\Lim_{|r| \rightarrow + \infty} f^{'}(r) > 2 \gamma > 0.
\end{equation}
We define the linear positive unbounded  operator $\Ag $ on
$L^2(\M)$ by:
\begin{equation}\label{s4}
\Ag \phi = - \Delta \phi + \gamma \phi, \ \forall \phi \in D(\Ag),
\end{equation}
where $$D(\Ag) = \left \{ \ro \in H^2(\M); \ \frac{\partial
\ro}{\partial \eta} = 0 \hbox{ on } \partial \M \right \}.
$$

Note that $ A_{\gamma}^{-1} $ is a compact linear operator on
$L^2(\M) $ and $| \Ag \cdot |_{L^2} $ is a norm on $D(\Ag) $ that is
equivalent to the $H^2$-norm.

We introduce the bilinear operators $B_0, $ $B_1$ (and their
associated trilinear forms $b_0, b_1 )$ as well as the coupling
mapping $R_0, $ which are defined from $D(A) \times D(A) $ into $H,
$  $D(A) \times D(\Ag) $ into $L^2(\M), $ and $L^2(\M) \times
D(A_{\gamma}^{3/2}) $ into $H, $ respectively. More precisely, we
set
\begin{eqnarray}
(B_0(u,v),w) = \Int_{\M} [(u \cdot \nabla) v] \cdot w \, dx =
b_0(u,v,w), \
\forall u, v, w \in D(A), \no\\
(B_1(u,\phi),\psi) = \Int_{\M} [(u \cdot \nabla) \phi]  \psi \, dx =
b_1(u,\phi,\psi), \ \forall u \in D(A), \ \phi, \psi \in D(\Ag), \no\\
(R_0(\mu,\phi),w) = \Int_{\M} \mu  [\nabla \phi \cdot w ] \, dx =
b_1(w,\phi,\mu), \ \forall w \in D(A), \ (\mu, \phi) \in L^2(\M)
\times D(A_{\gamma}^{3/2}).\no
\end{eqnarray}
Note that $$ R_0(\mu, \phi) = \P \mu \nabla \phi,$$ and
\beq\label{b0.1}
   |b_0(u,v,w)| \le c_b |u|^{1/2}_{L^2} \|u\|^{1/2} \|v\| |w|^{1/2}_{L^2} \|w\|^{1/2}, \quad\forall \, u, v, w  \in V,
\eeq
\beq\label{b0.2}
   |b_0(u,v,w)| \leq c_b |u|^{1/2}_{L^2} | A u|^{1/2}_{L^2} \| v \| |w|_{L^2},\\ \forall \,u \in D(A), \, v\in V, \, w  \in H,
\eeq
\begin{equation}\label{b0.3}
\quad |b_0(u,v,w)| \le c_b |u|^{1/2}_{L^2} \|u\|^{1/2} \|v\|^{1/2} |A v|^{1/2}_{L^2} |w|_{L^2},\\ \forall \, u \in V, v \in D(A), w \in H,
\eeq
\beq\label{b0.4}
   b_0(u,v,v)=0, \forall \, u,v \in V,
\eeq
the last equation implying
\begin{equation}\label{b0.5}
   b_0(u,v,w)=-b_0(u,w,v), \quad  \forall \, u,v,w \in V.
\end{equation}
Similar inequalities are valid for the trilinear form $b_1$:
\begin{equation}
   |b_1(u,\phi,\psi)| \le c_b |u|^{1/2}_{L^2} \|u\|^{1/2} \|\phi\| |\psi|^{1/2}_{L^2} \|\psi\|^{1/2},
  \forall \, u \in V, \phi, \psi \in H^1(\Omega),
\end{equation}
\begin{equation}\label{b2.1}
 \quad   |b_1(u,\phi,\psi)| \leq c_b |u|^{1/2}_{L^2} | A u|^{1/2}_{L^2} \| \phi \| |\psi|_{L^2},\forall \,u \in D(A), \, \phi\in H^1(\Omega), \, \psi  \in L^2(\Omega),
\end{equation}
\begin{eqnarray}\label{b3.1}
  |b_1(u,\phi,\psi)| \le c_b |u|^{1/2}_{L^2} \|u\|^{1/2} \|\phi\|^{1/2} |A\g \phi|^{1/2}_{L^2} |\psi|_{L^2},\\
  \forall \, u \in V, \phi \in D(A\g), \psi \in L^2(\Omega), \no
\end{eqnarray}
\begin{equation}\label{b4.1}
   b_1(u,\phi,\phi)=0,
   \quad b_1(u,\phi,f\g(\phi))=0, \quad \forall \, u \in V,\phi \in H^1(\Omega),
\end{equation}
\begin{equation}\label{b5.1}
   b_1(u,\phi,\psi)=-b_1(u,\psi,\phi), \quad \, \forall \, u \in V,\phi,\psi \in H^1(\Omega).
\end{equation}

Now we define the Hilbert spaces $ \Yo $ and $\Vo $ by
\begin{equation}\label{c1}
 \Yo = H \times H^1(\M),\ \ \Vo = V \times D(\Ag),
 \end{equation}
endowed with the scalar products whose associated norms are
\begin{equation}\label{c2}
\|(u,\phi) \|_{\Yo}^2 = \K^{-1} | u|^2_{L^2} + \epso (| \nabla \phi
|^2_{L^2}  + \gamma | \phi |^2_{L^2} )=:\K^{-1} | u|^2_{L^2} + \epso\| \phi \|\g^2,
 \end{equation}
 \begin{equation}\label{c21}
\ \ \Arrowvert (u,\phi)
\Arrowvert^2_{\Vo} = \Arrowvert u \Arrowvert^2 + | \Ag \phi
|^2_{L^2}.
  \end{equation}
 We also set
 \begin{equation}\label{fg}
  \fg (r) = f(r) - \alpha^{-1} \epso \gamma r
    \end{equation}
and observe that $\fg $ still satisfies (\ref{s3}) with $ \gamma $ in
place of $ 2 \gamma $ since, $ \epso \leq \alpha. $ Also, its
primitive, $\Fga(r) = \Int^r_0 \fg(\zeta) d \zeta $, is bounded from
below.

Throughout this article, we will denote by $c$ a generic positive
constant depending on the domain $\M. $

Using the notations above, we rewrite (\ref{ye6})--(\ref{b2}) as (see
\cite{gra1} for the details)
\begin{equation}\label{y6}
\left \{
\begin{array}{ll}
\Frac{d u}{d t} + \nuo A u + B_0(u,u) - \K R_0( \epso \Ag \phi,
\phi) = g, \ \hbox{
a.e. in } \M \times (0, + \infty), \\ \\
\mu = \epso \Ag \phi + \alpha \fg(\phi), \ \hbox{
a.e. in } \M \times (0, + \infty), \\ \\
\Frac{d \phi}{d t} + \mu + B_1(u, \phi) = 0, \ \hbox{ a.e. in } \M
\times (0, + \infty).
\end{array} \right. \end{equation}

The weak formulation of (\ref{y6}), (\ref{ye6}) was proposed and studied
in \cite{gra1,gra2}, and the existence and uniqueness of solution
was proved.

\section{A time discretization of (\ref{y6})}

In this article we consider a time discretization of (\ref{y6})
using the fully implicit Euler scheme,
\begin{equation}\label{q-ies}
\left \{
\begin{array}{ll}
\frac{u^n - u^{n-1}}{{\dt}}+ \nuo A u^n + B_0(u^n,u^n) - \K R_0( \epso \Ag \phi^n,
\phi^n) = g^n,  \\
\mu^n = \epso \Ag \phi^n + \alpha \fg(\phi^n), \\
\frac{\phi^n - \phi^{n-1}}{{\dt}} + \mu^n + B_1(u^n, \phi^n) = 0,\\
u^0 = u_0, \phi^0=\phi_0,
\end{array} \right.
\end{equation}
and prove that the attractors generated by the above system converge
to the attractor generated by the continuous system (\ref{y6}) as
the time-step converges to zero. To prove the existence of the
discrete attractors we need to use the theory of the multi-valued
attractors, that we discuss in Subsection \ref{ssec:abs}.

Throughout the article, we assume that $g \in L^\infty(\R_+; H)$ and
we let $\|g\|_{\infty} := \|g\|_{L^\infty(\R_+; H) }$.

\subsection{$\Yo$-Uniform Boundedness }
We begin with one of our main results, which proves the uniform
boundedness of the approximate solution $(u^n, \phi^n)$ in $\Yo$.
Once the $\Yo$-uniform stability is established, the $\Vo$-uniform
boundedness follows right away (see Proposition \ref{t:bdh} below).

\begin{Thm}\label{t:bdh}
Let $(u^n, \phi^n)$ be a solution of (\ref{q-ies}). Then there exists $\kappa>0$ such that for every ${\dt}>0$, we have
\beq\label{q:bdh}
  \|(u^n,\phi^n) \|_{\Yo}^2 \le (1+\kappa \dt)^{-n} Q^2(\|(u_0, \phi_0)\|_{\Yo})+\rho_0^2\left[1- \left(1+\kappa {\dt}\right)^{-n} \right],
    \>\forall\, n\ge0,
\eeq
where the monotonically increasing function $Q$ is independent of $n$, and $\rho_0$ (given in
(\ref{rho}) below), is independent of the initial data.

Moreover, there exists $K_1=K_1(\|(u_0, \phi_0)\|_{\Yo},\gvi)$ such that for every ${\dt}>0$, we have
\beq\label{q:bdv}
  \|(u^n,\phi^n) \|_{\Yo} \le K_1,
    \>\forall\, n\ge0,
\eeq and for every $i=1, \cdots, n$ there exist $M_1=M_1(\|(u^{i-1},
\phi^{i-1})\|_{\Yo}, \gvi, (n-i+1)k)$ and $M_2=M_2(\|(u^{i-1},
\phi^{i-1})\|_{\Yo}, \gvi, (n-i+1)k)$, increasing in their
arguments, such that
\beq\label{M1}
   k \sum_{j=i}^n \left( \frac{\nuo}{2\K} \|u^n\|^2+ 2  |\mu^n|^2_{L^2}\right) \le  M_1,
\eeq
\beq\label{M2}
   k \sum_{j=i}^n |\Ag(\phi^j)|^2_{L^2} \le  M_2.
\eeq
\end{Thm}

\begin{proof}
Taking the scalar product of the first equation of  (\ref{q-ies}) with $2 {\dt} u^n$ in $L^2$
and using the relation
\begin{equation}
      2(\varphi - \psi, \varphi)_{L^2}=|\varphi|^2_{L^2}-|\psi|^2_{L^2}+|\varphi-\psi|^2_{L^2},
\eeq
and the skew property (\ref{b0.4}), we obtain
\begin{eqnarray}
      |u^n|^2_{L^2} &-& |u^{n-1}|^2_{L^2} + |u^n - u^{n-1}|^2_{L^2}
         + 2 \nuo {\dt} \|u^n\|^2 \label{2.23}\\
         & -& 2 \K \dt b_1(u^n, \phi^n, \epso \Ag \phi^n)= 2  {\dt} (g^n,u^n)\l. \nonumber
\end{eqnarray}

Using the second equations of  (\ref{q-ies}) and of (\ref{b4.1}), we
have $b_1(u^n, \phi^n, \epso \Ag \phi^n)=b_1(u^n, \phi^n,\mu^n)$
and thus (\ref{2.23}) becomes
\begin{eqnarray}
      |u^n|^2_{L^2} &-& |u^{n-1}|^2_{L^2} + |u^n - u^{n-1}|^2_{L^2}
         + 2 \nuo {\dt} \|u^n\|^2 \label{2.25}\\
         &-&2 \K \dt b_1(u^n, \phi^n, \mu^n)= 2  {\dt} (g^n,u^n)\l. \nonumber
\end{eqnarray}

Multiplying the third equation of (\ref{q-ies}) by $2 {\dt} \mu^n$ and integrating we obtain
\beq\label{2.26}
2(\phi^n-\phi^{n-1}, \mu^n)\l+2\dt |\mu^n|^2_{L^2}+2\dt b_1(u^n, \phi^n, \mu^n)=0.
\eeq
Dividing (\ref{2.25}) by $\K$ and adding the resulting equation to (\ref{2.26}), we find
\begin{eqnarray}
     &\frac{1}{\K}& \left[|u^n|^2_{L^2} - |u^{n-1}|^2_{L^2}+  |u^n - u^{n-1}|^2_{L^2} \right]
        +\frac{ 2 \nuo}{\K} {\dt} \|u^n\|^2 \nonumber\\
         & +&2(\phi^n-\phi^{n-1}, \mu^n)\l+2\dt |\mu^n|^2_{L^2}=  \frac{2}{\K} {\dt} (g^n,u^n)\l. \label{2.27}
\end{eqnarray}
Using the second equation of  (\ref{q-ies}), (\ref{s4}) and (\ref{fg}), we obtain
\begin{eqnarray}
2(\phi^n-\phi^{n-1}, \mu^n)\l=\epso \left(\|\ph\|\g^2-\|\phn\|\g^2+\|\ph-\phn\|\g^2 \right) + 2 \alpha (\ph-\phn, f\g(\ph))\l. \no
\end{eqnarray}
\begin{eqnarray}\label{2.30}
2 \alpha (\ph-\phn, f\g (\ph))\l
&=&2 \alpha  \F\g(\ph)-2 \alpha\F\g(\phn)+2 \alpha \R\g ^n,
\end{eqnarray}
where
\begin{equation}\label{2.31}
\F\g(\ph)=\int_\Omega F\g (\phi^n(x)) dx,
\end{equation}
\begin{equation}\label{2.32}
\R\g ^n=-\int_\Omega \int_0^1 \left[f\g \left(\phi^{n-1}(x)+t(\phi^n(x)-\phi^{n-1}(x))\right)-f\g (\phi^n(x))\right](\phi^n(x)-\phi^{n-1}(x)) \, dt \, dx,
\end{equation}
and thus
\begin{eqnarray}
2(\phi^n-\phi^{n-1}, \mu^n)\l = &\epso& \left(\|\ph\|\g ^2-\|\phn\|\g ^2+\|\ph-\phn\|\g ^2 \right)+  2 \alpha  \F\g(\ph)-2 \alpha\F\g(\phn)+2 \alpha \R\g ^n. \nonumber
\end{eqnarray}
Combining the above relation with (\ref{2.27}) we find
\begin{eqnarray}
    & \qquad \qquad \frac{1}{\K}& \left[|u^n|^2_{L^2} - |u^{n-1}|^2_{L^2} +  |u^n - u^{n-1}|^2_{L^2} \right]
          +\epso \left(\|\ph\|\g ^2-\|\phn\|\g ^2+\|\ph-\phn\|\g ^2 \right)\nonumber \\
         & \qquad \qquad & +\frac{ 2 \nuo}{\K} {\dt} \|u^n\|^2+ 2 \alpha  \F\g(\ph)-2 \alpha\F\g(\phn)+2 \alpha \R\g ^n+2\dt |\mu^n|^2_{L^2}=  \frac{2}{\K} {\dt} (g^n,u^n)\l. \label{2.34}
\end{eqnarray}

Multiplying the third equation of (\ref{q-ies}) by $2 {\dt} \phi^n$ and integrating we obtain (using the second equation of (\ref{q-ies}))
\begin{equation}
 \qquad |\ph|^2_{L^2}-|\phn|^2_{L^2}+|\ph-\phn|^2_{L^2} + 2 \dt \epso \|\phi^n\|^2\g + 2 \alpha \dt (\fg(\phi^n), \phi^n)\l=0.\label{2.36}
\end{equation}
Adding (\ref{2.34}) and (\ref{2.36}) we find
\begin{eqnarray}
     &\frac{1}{\K}\left[|u^n|^2_{L^2} - |u^{n-1}|^2_{L^2} +  |u^n - u^{n-1}|^2_{L^2} \right]+\epso \left(\|\ph\|\g^2-\|\phn\|\g^2+\|\ph-\phn\|\g^2 \right) \nonumber\\
          &  \qquad \qquad  \qquad+ |\ph|^2_{L^2}-|\phn|^2_{L^2}+|\ph-\phn|^2_{L^2}+ 2 \alpha  \F\g(\ph)-2 \alpha\F\g(\phn) \label{2.37}\\
         & \qquad+ \frac{ 2 \nuo}{\K} {\dt} \|u^n\|^2+2 \dt \epso \|\phi^n\|^2\g +2\dt |\mu^n|^2_{L^2}+ 2 \alpha \dt (\fg(\phi^n), \phi^n)\l +2 \alpha \R\g^n =  \frac{2}{\K} {\dt} (g^n,u^n)\l. \nonumber
\end{eqnarray}
Using the Cauchy--Schwarz inequality and the Poincar\'e inequality
(\ref{Poin}), we majorize the right-hand side of (\ref{2.37}) by
\begin{eqnarray}
     \frac{2}{\K} {\dt} (g^n,u^n)\l
      & \leq& \frac{2}{\K} {\dt} |g^n|\l |u^n|\l \leq \frac{2 }{\K} \sqrt {c_\Omega}{\dt} |g^n|\l \|u^n\| \nonumber\\
      & \leq& \frac{\nuo}{\K} {\dt}\|u^n\|^2 + \frac{c_\Omega}{\nuo \K} \dt |g^n|_{L^2}^2. \label{2.38}
\end{eqnarray}

Hereafter, we assume that the potential function $f $ satisfies the
following  additional condition:
\begin{equation}\label{cf1}
f'(r)\geq -\frac{1}{2\alpha}, \ \forall r \in \R.
\end{equation}

Now, using the mean value theorem and recalling (\ref{fg}), as
well as (\ref{cf1}), relation (\ref{2.32}) yields
\begin{eqnarray}\label{2.39}
2 \alpha \R\g^n \geq -\frac{1}{2} |\phi^n-\phi^{n-1}|^2_{L^2} - \epso \gamma |\phi^n-\phi^{n-1}|^2_{L^2}.
\end{eqnarray}

Relations (\ref{2.37}), (\ref{2.38}), (\ref{2.39}) and (\ref{c2})
give
\begin{eqnarray}
 \frac{1}{\K} &\left[|u^n|^2_{L^2} - |u^{n-1}|^2_{L^2} +  |u^n - u^{n-1}|^2_{L^2} \right]+\epso \|\ph\|\g^2- \epso\|\phn\|\g^2+\epso \|\ph-\phn\|^2  \nonumber\\
          & +|\ph|^2_{L^2}-|\phn|^2_{L^2}+\frac{1}{2}|\ph-\phn|^2_{L^2}+ 2 \alpha  \F\g(\ph)-2 \alpha\F\g(\phn) \label{2.431}\\
         & +\frac{ \nuo}{\K} {\dt} \|u^n\|^2+2 \dt \epso \|\phi^n\|^2\g +2\dt |\mu^n|^2_{L^2}+ 2 \alpha \dt (\fg(\phi^n), \phi^n)\l  \leq \frac{c_\Omega}{\nuo \K} \dt |g^n|_{L^2}^2.  \nonumber
\end{eqnarray}

Now, for any $n\geq 1$, let
\begin{equation}\label{2.44}
E^n=\frac{1}{\K}|u^n|^2_{L^2}+\epso \|\ph\|\g^2+2 \alpha  \F\g(\ph) +|\ph|^2_{L^2}+2 \alpha C_{F\g}|\Omega|,
\end{equation}
where $C_{F\g}$ is taken large enough to ensure that $E^n\geq 0$
(recall that $F\g$ is bounded from below by a constant independent
of $\nuo$ and $\alpha$). We rewrite (\ref{2.431}) in the form
\begin{equation}\label{2.451}
E^n-E^{n-1}+\kappa \dt E^{n}\leq \dt \Lambda^n,
\end{equation}
where $\kappa \in (0,1)$ is to be determined and
\begin{eqnarray}
\Lambda^n& =&-\frac{ \nuo}{\K} \|u^n\|^2+ \frac{\kappa}{\K} |u^n|^2_{L^2}-(2-\kappa) \epso \|\phi^n\|^2\g+\kappa|\ph|^2_{L^2}+2 \alpha \kappa C_{F\g}|\Omega|  \nonumber\\
& -& 2 |\mu^n|^2_{L^2} +\frac{c_\Omega}{\nuo \K} |g^n|_{L^2}^2+2 \alpha \left[\kappa
\left(F\g(\ph)-\fg(\phi^n)\phi^n,1\right)\l-(1-\kappa)\left(\fg(\phi^n)\phi^n,1\right)\l\right].  \nonumber
\end{eqnarray}

Now note that owing to assumption  (\ref{s1}), we have (for any $r\in \R$)
\begin{equation}\label{2.40}
f\g(r)r\geq \frac{c_\star}{2}|f\g(r)|(1+|r|)-\frac{c_f}{2}(1+\alpha^{-1}\epso),\no
\end{equation}
\begin{equation}\label{2.41}
F\g(r)-f\g(r)r \leq c_f'(1+\alpha^{-1}\epso)|r|^2+c_f'', \no
\end{equation}
\begin{equation}\label{2.42}
|F\g(r)| \leq |f\g(r)|(1+|r|)+c_1, \no
\end{equation}
where $c_f, c_\star, c_f', c_f'', c_1$ are positive, sufficiently large constants that depend on $f$ only.

Using the above inequalities and the Poincar\'e inequality (\ref{Poin}), we obtain the following bound on $\Lambda^n$:
\begin{eqnarray}\label{2.45}
\Lambda^n &\leq&
-\frac{ 1}{\K}(\nuo-\kappa c_\Omega) \|u^n\|^2-(2-\kappa) \epso \|\phi^n\|^2-\left[2-\frac{\kappa}{\epso \gamma}\left(1+\epso \gamma + 2c_f'(\alpha+\epso)\right)\right]\epso\gamma|\ph|^2_{L^2} \nonumber\\
& & \qquad \qquad -2 |\mu^n|^2_{L^2} +\frac{c_\Omega}{\nuo \K}
|g^n|_{L^2}^2-(1-\kappa)c_\star \alpha (|F\g(\phi^n)|,1)\l+c_2,
\end{eqnarray}
where
\begin{equation}
c_2=\left[(1-\kappa)(c_1c_\star \alpha +c_f \alpha +c_f\epso) +2 \alpha \kappa C_{F\g}+2 \alpha \kappa  c_f''\right]|\Omega|.
\end{equation}
Choosing $\kappa\in (0,1)$ as
\begin{equation}\label{2.47}
\kappa =\min\left\{\frac{\nuo}{2c_\Omega},\frac{\epso \gamma}{1+\epso \gamma + 2c_f'(\alpha+\epso)}\right\},
\end{equation}
relation (\ref{2.45}) gives
\begin{eqnarray}\label{2.48}
\Lambda^n &\leq&  -\frac{\nuo}{2\K} \|u^n\|^2- \epso \|\phi^n\|^2\g
-2 |\mu^n|^2_{L^2} +\frac{c_\Omega}{\nuo \K}
|g^n|_{L^2}^2-(1-\kappa)c_\star \alpha (|F\g(\phi^n)|,1)\l+c_2,\no
\end{eqnarray}
and combining it with (\ref{2.451}), we obtain
\begin{eqnarray}
E^n&-&E^{n-1}+\kappa \dt E^{n}+\frac{1}{2} \dt \left( \frac{\nuo}{\K} \|u^n\|^2+ \epso \|\phi^n\|^2\g\right)+ 2 \dt |\mu^n|^2_{L^2} \nonumber\\
 & +&c_3 \dt|F\g(\phi^n)|_{L^1} \leq \frac{c_\Omega}{\nuo \K} \dt |g^n|_{L^2}^2+c_2 \dt. \label{2.49}
\end{eqnarray}
Neglecting some positive terms, we obtain
\begin{equation}\label{2.50}
      E^n \leq \frac{1}{\beta}E^{n-1}
         + \frac{1}{\beta}\dt \left(\frac{c_\Omega}{\nuo \K} |g^n|_{L^2}^2+c_2\right),
\end{equation}
where
\begin{equation}\label{2.51}
      \beta = 1+ \kappa {\dt}.
\end{equation}
Using recursively (\ref{2.50}), we find
\begin{eqnarray}\label{2.52}
 E^n &\leq& \frac{1}{\beta^n}E^0
                        + \dt
                        \sum_{i=1}^{n}\frac{1}{\beta^i}\left(\frac{c_\Omega}{\nuo \K} |g^{n+1-i}|_{L^2}^2 +c_2\right)\\
                  &\leq& \left(1+ \kappa {\dt} \right)^{-n}E^0
                        + \frac{1}{ \kappa}\left(\frac{c_\Omega}{\nuo \K} \gvi^2+c_2\right)
                         \left[1- \left(1+\kappa {\dt}\right)^{-n} \right]. \nonumber
\end{eqnarray}
Now observe that, due to (\ref{s1}), we can find $C_f>0$ such that
\begin{equation}\label{2.53}
E^n\leq C_f\left(1+\|(u^n, \phi^n)\|_{\Yo}^2+|\phi^n|^{m+2}_{L^{m+2}}\right),
\end{equation}
and thus, relation (\ref{2.52}) yields
\begin{equation}\label{2.54}
E^n \leq (1+\kappa \dt)^{-n} Q^2(\|(u_0, \phi_0)\|_{\Yo})+\rho_0^2\left[1- \left(1+\kappa {\dt}\right)^{-n} \right],
\end{equation}
where $Q$ is a monotonically increasing function of the initial data, independent of $n$, and
\begin{equation}\label{rho}
\rho_0^2=\frac{1}{ \kappa}\left(\frac{c_\Omega}{\nuo \K}
\gvi^2+c_2\right).
\end{equation}
We therefore obtain
\begin{equation}\label{2.57}
E^n \leq K_1^2=K_1^2(\|(u_0, \phi_0)\|_{\Yo},\gvi):=Q^2(\|(u_0, \phi_0)\|_{\Yo})+\rho_0^2, \, \forall  n\ge0,
\end{equation}
and since $\|(u^n, \phi^n)\|_{\Yo}^2\leq E^n$, relations (\ref{2.54}) and (\ref{2.57}) give (\ref{q:bdh}) and (\ref{q:bdv}), respectively.

Adding inequalities (\ref{2.49}) with $n$ from $i$ to $N$, dropping some positive terms,
and recalling (\ref{2.53}), we obtain conclusion (\ref{M1}) of the theorem.

Now using (\ref{q-ies}), (\ref{fg}), (\ref{s2})  and the Sobolev imbedding $H^1(\Omega) \hookrightarrow L^{2m+2}(\Omega)$, for any $m$, we obtain
\begin{eqnarray}\label{2.58}
|\Ag \phi^n|^2_{L^2}&\leq& \frac{3}{\nu_2^2} |\mu^n|^2_{L^2}  + \frac{6\beta^2}{\nu_2^2} c_f^2  |\Omega|+ \frac{6\beta^2}{\nu_2^2} c_f''  (\|\phi^n\|\g^2)^{m+1}+3 \gamma^2 |\phi^n|^2_{L^2},
 \end{eqnarray}
for some positive constant $c_f''$ depending on $c_f$. Summing with $n$ from $i$ to $N$ and using (\ref{M1}) and  (\ref{2.57}), we obtain conclusion (\ref{M2}).
This completes the proof of the theorem.
\end{proof}


As a direct consequence of (\ref{q:bdh}), we have

\begin{Cor}\label{C1}
If
\beq\label{q:k1}
   0 < \dt \leq \frac{1}{\kappa },
\eeq then  $B_{\Yo}(0, \sqrt 2 \rho_0)$, the ball in $\Yo$ centered
at $0$ and radius $ \sqrt 2\rho_0$, is an absorbing ball for
$({u}^n, \phi^n)$ in $\Yo$.
\end{Cor}


\subsection{$\Vo$-Uniform Boundedness}

We now seek to obtain uniform bounds for $(u^n,\phi^n)$ in $\Vo$,
similar to those we have already obtained in $\Yo$ (see
(\ref{q:bdh}) above). In order to do this, we will first use the
discrete Gronwall lemma to derive an upper bound on
$\|(u^n,\phi^n)\|_{\Vo}$, $n \leq N$, for some $N >0$, and then we
will use the discrete uniform Gronwall lemma  to obtain an upper
bound on $\|(u^n,\phi^n)\|_{\Vo}$, $n \geq N$.

We begin with some preliminary inequalities.


\begin{Lem}\label{l2}
For every ${\dt}>0$, we have
\beq\label{1.71}
   \|(u^n,\phi^n)\|^2_{\Vo} \leq K_2 \|(u^{n-1},\phi^{n-1})\|^2_{\Vo}  + c_3(|f\g(\phi^n)|^2_{L^2}+|g^n|^2_{L^2}), \, \forall n\geq 1,
\eeq
where $K_2=K_2(\|(u_0, \phi_0)\|_{\Yo},\gvi)$ and $c_3>0$ are given below, in (\ref{1.81}) and (\ref{1.82}), respectively.
\end{Lem}

\begin{proof}
Multiplying the first equation of (\ref{q-ies}) by $2/\K \,{\dt}
(u^n-u^{n-1})$, the third equation of (\ref{q-ies}) by $2 \epso {\dt}A\g( \phi^n-\phi^{n-1})$, adding the resulting equations  and integrating, we obtain (using the second equation of (\ref{q-ies}))
\begin{eqnarray}\label{1.74}
     & & \qquad \qquad \qquad\frac{2}{\K}|u^n - u^{n-1}|^2_{L^2} +2 \epso\|\phi^n-\phi^{n-1}\|^2\g + \frac{ \nuo}{\K} {\dt} \|u^n\|^2 +\nu_2^2{\dt} |A\g\phi^n|^2_{L^2} \\
     &  - & \frac{ \nuo}{\K} \dt \|u^{n-1}\|^2 - \nu_2^2{\dt} |A\g\phi^{n-1}|^2_{L^2} +\frac{ \nuo}{\K} \dt \|u^n - u^{n-1}\|^2 +\nu_2^2{\dt} |A\g(\phi^n-\phi^{n-1})|^2_{L^2} \no\\
      & + & \frac{ 2}{\K}\dt\, b_0(u^n,u^n,u^n-u^{n-1}) - 2 \dt b_1(u^n-u^{n-1}, \phi^n, \epso \Ag \phi^n)+ 2 \epso \dt b_1(u^n, \phi^n, A\g( \phi^n-\phi^{n-1})) \no\\
      &=&  \frac{ 2}{\K}  {\dt} (g^n,u^n- u^{n-1})\l - 2 \epso \alpha \dt (\fg(\phi^n), A\g( \phi^n-\phi^{n-1}))\l. \no
\end{eqnarray}
Using the Cauchy--Schwarz inequality, we majorize the right-hand
side of (\ref{1.74}) by
\begin{eqnarray}\label{1.75}
      \frac{ 2}{\K}  {\dt} (g^n,u^n- u^{n-1})\l
       & \leq & \frac{ 2}{\K}  {\dt} |g^n|\l |u^n- u^{n-1}|\l \no\\
       & \leq & \frac{ 2}{\K} \sqrt{c_\Omega} \, {\dt} |g^n|\l \|u^n- u^{n-1}\| \quad (\textrm{by (\ref{Poin}}))\\
       & \leq & \frac{ \nuo}{4\K} \dt \|u^n- u^{n-1}\|^2  + \frac{ 4 c_\Omega}{ \nuo\K}  {\dt}  |g^n|^2_{L^2}, \no
\end{eqnarray}
\begin{eqnarray}\label{1.76}
2 \epso \alpha  \dt |(\fg(\phi^n), A\g( \phi^n-\phi^{n-1}))\l| &\leq& 2 \epso \alpha \dt |\fg(\phi^n)|\l | A\g( \phi^n-\phi^{n-1})|\l \no\\
& \leq & \frac{\nu_2^2}{4} \dt | A\g( \phi^n-\phi^{n-1})|^2_{L^2} + 4 \alpha ^2 \dt |\fg(\phi^n)|^2_{L^2}.
\end{eqnarray}
The nonlinear terms are bounded as follows:
\begin{eqnarray}\label{1.77}
 \frac{ 2}{\K}  \dt b_0(u^n,u^n,u^n-u^{n-1}) = -\frac{ 2}{\K} \dt b_0(u^n,u^n,u^{n-1}) \quad (\textrm{by (\ref{b0.4}})) \\
  \le \frac{ 2}{\K}   c_b \dt |u^n|_{L^2} \|u^n\| \|u^{n-1}\| \quad (\textrm{by (\ref{b0.1}}))
\le  \frac{\nu_1}{2\K} \dt \|u^n\|^2
   + \frac{2 c_b^2}{\nu_1} K_1^2 \dt \|u^{n-1}\|^2 \quad (\textrm{by } (\ref{q:bdv})), \no
\end{eqnarray}
\begin{eqnarray}\label{1.78}
- 2 \dt &b_1&(u^n-u^{n-1}, \phi^n, \epso \Ag \phi^n)+ 2 \epso \dt b_1(u^n, \phi^n, A\g( \phi^n-\phi^{n-1})) \no \\
& =& - 2\epso \dt b_1(u^n-u^{n-1}, \phi^n,  \Ag \phi^{n-1})+ 2 \epso \dt b_1(u^{n-1}, \phi^n, A\g( \phi^n-\phi^{n-1}))\no\\
& \leq & 2\epso c_b \dt |u^n-u^{n-1}|^{1/2}_{L^2} \|u^n-u^{n-1}\|^{1/2}\|\phi^n\|^{1/2} |A\g \phi^n|^{1/2}_{L^2} |A\g \phi^{n-1}|_{L^2}\no\\
& \quad & + 2\epso c_b \dt |u^{n-1}|^{1/2}_{L^2} \|u^{n-1}\|^{1/2}\|\phi^n\|^{1/2} |A\g \phi^n|^{1/2}_{L^2} |A\g( \phi^n-\phi^{n-1})|_{L^2}
\quad (\textrm{by (\ref{b3.1}}))\\
& \leq & \frac{ \nuo}{4\K} \dt \|u^n - u^{n-1}\|^2 + \frac{ \nu_2^2}{2}{\dt} |A\g\phi^n|^2_{L^2} + c K_1^2 \dt  |A\g \phi^{n-1}|_{L^2}^2\no\\
& \quad & +\frac{ \nu_2^2}{4}{\dt} |A\g( \phi^n-\phi^{n-1})|_{L^2}^2+ c K_1^4 \dt \|u^{n-1}\|^2 \quad (\textrm{by } (\ref{q:bdv})). \no
\end{eqnarray}
Relations (\ref{1.74})--(\ref{1.78}) yield
\begin{eqnarray}\label{1.79}
     && \qquad \frac{2}{\K} |u^n - u^{n-1}|^2_{L^2} +2 \epso\|\phi^n-\phi^{n-1}\|^2\g + \frac{\nu_1}{2\K} \dt \|u^n\|^2 +\frac{\nu_2^2}{2}{\dt} |A\g\phi^n|^2_{L^2} \no\\
     & & \qquad-\left(\frac{ \nuo}{\K} +\frac{2 c_b^2}{\nu_1} K_1^2  + c K_1^4 \right)\dt \|u^{n-1}\|^2  - (\nu_2^2 + c K_1^2) {\dt} |A\g\phi^{n-1}|^2_{L^2}+\frac{ \nuo}{2\K} \dt \|u^n - u^{n-1}\|^2 \no\\
      &  & \qquad \qquad+\frac{\nu_2^2}{2}{\dt} |A\g(\phi^n-\phi^{n-1})|^2_{L^2} \leq \frac{ 4 c_\Omega}{ \nuo\K}  {\dt}  |g^n|^2_{L^2}+ 4 \alpha ^2 \dt |\fg(\phi^n)|^2_{L^2},
\end{eqnarray}
and neglecting some positive terms we obtain
\begin{eqnarray}\label{1.80}
     && \frac{\nu_1}{2\K} \|u^n\|^2 +\frac{\nu_2^2}{2}|A\g\phi^n|^2_{L^2} \leq \left(\frac{ \nuo}{\K} +\frac{2 c_b^2}{\nu_1} K_1^2  + c K_1^4 \right) \|u^{n-1}\|^2 \no\\
     && + (\nu_2^2 + c K_1^2) |A\g\phi^{n-1}|^2_{L^2}+\frac{ 4 c_\Omega}{ \nuo\K} |g^n|^2_{L^2}+ 4 \alpha ^2 |\fg(\phi^n)|^2_{L^2}.
\end{eqnarray}
Taking
\beq\label{1.81}
K_2=2\left( \frac{\K}{\nu_1}
+\frac{1}{\nu_2^2}\right)\left(\frac{ \nuo}{\K} +\frac{2
c_b^2}{\nu_1} K_1^2  + c K_1^4+\nu_2^2 + c K_1^2 \right),
\eeq
and
\beq\label{1.82}
c_3=2\left( \frac{\K}{\nu_1}
+\frac{1}{\nu_2^2}\right)\left( \frac{ 4 c_\Omega}{ \nuo\K} +4
\alpha^2\right),
\eeq
we obtain conclusion (\ref{1.71}) of the lemma.
\end{proof}


\begin{Lem}\label{l3}
For every ${\dt}>0$ and for every $n \geq 1$ we have
\begin{eqnarray}\label{1.83}
  c_4 K_1^2 &{\dt}\Arrowvert (u^n,\phi^n)
\Arrowvert^4_{\Vo}- \Arrowvert (u^n,\phi^n)\Arrowvert^2_{\Vo}
     +\Arrowvert (u^{n-1},\phi^{n-1})
\Arrowvert^2_{\Vo} \no\\
& +  c_5 \dt \left(| \fg'(\phi^n) \nabla \phi^n|^2_{L^2}+|g^n|_{L^2}^2\right)\geq 0,
\end{eqnarray} for some positive constants $c_4$ and $c_5$.
\end{Lem}

\begin{proof}

Multiplying the first equation of (\ref{q-ies}) by $2 \dt Au^n$, the third equation of (\ref{q-ies}) by $2 {\dt} A\g^2
\phi^n$, adding the resulting equations and
integrating, we obtain (using the second equation of (\ref{q-ies}))
\begin{eqnarray}\label{2.75}
      \|u^n\|^2 + |A\g\ph|^2_{L^2}- (\|u^{n-1}\|^2+|A\g\phn|^2_{L^2})+ \|u^n - u^{n-1}\|^2 +|A\g(\ph-\phn)|^2_{L^2}\no\\
         + 2 \nuo {\dt} |Au^n|^2_{L^2} + 2 \dt \epso |\Ag ^{3/2}\phi^n|^2_{L^2}+2 \dt b_0(u^n, u^n, Au^n)-2 \K \dt b_1(Au^n, \phi^n, \epso \Ag \phi^n)\no\\
         +2 \dt b_1(u^n, \phi^n,A\g^2 \phi^n)= 2  {\dt} (g^n,Au^n)\l-2\alpha \dt (\fg(\phi^n),A\g^2 \phi^n)\l.
\end{eqnarray}
Using the Cauchy-Schwarz inequality, we estimate the right-hand side of the above equality as
\begin{equation}\label{2.76}
2  {\dt} (g^n,Au^n)\l \leq 2  {\dt} |g^n|\l|Au^n|\l \leq
\frac{\nuo}{4} \dt |Au^n|^2_{L^2}+\frac{4}{\nuo} \dt|g^n|_{L^2}^2,
\end{equation}
\begin{eqnarray}\label{2.77}
2\alpha \dt |(\fg(\phi^n),A\g^2 \phi^n)\l| =2\alpha \dt |(A\g^{1/2}\fg(\phi^n),A\g^{3/2} \phi^n)\l|\leq 2\alpha \dt |A\g^{1/2}\fg(\phi^n)|\l |A\g^{3/2} \phi^n|\l\no\\
 \leq c |\nabla \fg(\phi^n)|\l |A\g^{3/2} \phi^n|\l \leq \frac{\epso}{3} \dt |A\g^{3/2} \phi^n|^2_{L^2} + c \dt | \fg'(\phi^n) \nabla \phi^n|^2_{L^2}.
\end{eqnarray}
We bound the nonlinear terms as follows:
\begin{eqnarray}\label{2.78}
2 \dt |b_0(u^n, u^n, Au^n)| \leq 2 c_b  \dt  |u^n|^{1/2}_{L^2}
\|u^n\| |A u^n|^{3/2}_{L^2} \\
\leq \frac{\nuo}{4} \dt |A u^n|^2_{L^2}+ c \dt |u^n|^2_{L^2}  \|u^n\|^4, \no
\end{eqnarray}
\begin{eqnarray}\label{2.79}
2 \K \dt |b_1(Au^n, \phi^n, \epso \Ag \phi^n)|\leq 2 \epso \K  \dt |Au^n|_{L^2} |\nabla \phi^n|_{L^\infty} |\Ag \phi^n|_{L^2} \no\\
\leq c \epso \K  \dt |Au^n|_{L^2} \|\phi^n\|^{1/2} |A\g^{3/2} \phi^n|^{1/2}_{L^2} |A\g \phi^n|_{L^2} \no \\
 \quad (\textrm{by Agmon's inequality}))\\
\leq \frac{\epso}{3} \dt |A\g^{3/2} \phi^n|^2_{L^2}+ \frac{\nuo}{4} \dt |A u^n|^2_{L^2}+ c \dt \|\phi^n\|^2 |\Ag \phi^n|_{L^2}^4, \no
\end{eqnarray}
\begin{eqnarray}\label{2.80}
2 \dt |b_1(u^n, \phi^n,A\g^2 \phi^n)|&=&2 \dt |(A\g^{1/2}B_1(u^n, \phi^n),A\g^{3/2} \phi^n)\l| \leq 2 \dt |A\g^{1/2}B_1(u^n, \phi^n)|\l |A\g^{3/2} \phi^n|\l \no\\
 &\leq& c \dt \|u^n\|^{1/2} |A u^n|^{1/2}_{L^2} \|\phi^n\|^{1/2} |A\g \phi^n|_{L^2}^{1/2} |A\g^{3/2} \phi^n|_{L^2}+ c \dt |u^n|^{1/2}_{L^2} \|u^n\|^{1/2}|A\g \phi^n|_{L^2}^{1/2} |A\g^{3/2} \phi^n|_{L^2}^{3/2}\no\\
&&\qquad \qquad \qquad(\textrm{by Ladyzhenskaya's inequality}))\\
&\leq& \frac{\epso}{3} \dt |A\g^{3/2} \phi^n|^2_{L^2}+\frac{\nuo }{4} \dt |A u^n|^2_{L^2} +c \dt\|u^n\|^2 \|\phi^n\|^2 |A\g \phi^n|_{L^2}^2
+c \dt |u^n|^{2}_{L^2} \|u^n\|^{2}|A\g \phi^n|_{L^2}^{2}.\no
\end{eqnarray}

Recalling (\ref{q:bdv}), relations (\ref{2.75})--(\ref{2.80}) give
\begin{eqnarray}\label{2.81}
      \|u^n\|^2 + |A\g\ph|^2_{L^2}- (\|u^{n-1}\|^2+|A\g\phn|^2_{L^2})+ \|u^n - u^{n-1}\|^2 +|A\g(\ph-\phn)|^2_{L^2}\no\\
         +\nuo \dt |A u^n|^2_{L^2} +\epso \dt  |\Ag ^{3/2}\phi^n|^2_{L^2} \leq c K_1^2 \dt \|u^n\|^4 +c K_1^2 \dt  |\Ag \phi^n|_{L^2}^4\\
         +c K_1^2 \dt\|u^n\|^2|A\g \phi^n|_{L^2}^2+ c \dt | \fg'(\phi^n) \nabla \phi^n|^2_{L^2} +\frac{2}{\nuo} \dt|g^n|_{L^2}^2,\no
\end{eqnarray}
from which the conclusion of the lemma follows right away.
\end{proof}

In order to prove the uniform boundedness of
$\|(u^n,\phi^n)\|_{\Vo}$ we will make use of the following two
lemmas, whose proofs can be found in \cite{shen:90}:

\begin{Lem}\label{t:dgronwall}
Given $\dt>0$ and positive sequences $\x_n$, $\eta_n$ and $\zeta_n$ such
that
\begin{equation}\label{q:gronseq}
  \x_n \le \x_{n-1} (1 + \dt \eta_{n-1}) + \dt \zeta_n,
    \qquad\textrm{for}\> n \geq 1,
\end{equation}
we have, for any $n \geq 2$,
\begin{equation}\label{q:gronest}
  \x_n \le \left(\x_0 + \sum_{i=1}^{n} \dt \zeta_i \right) \exp\biggl( \sum_{i=0}^{n-1} \dt \eta_i \biggr).
\end{equation}
\end{Lem}


\begin{Lem}\label{t:dugronwall}
Given $\dt>0$, a positive integer $ n_0$, positive sequences $\x_n$,
$\eta_n$ and $\zeta_n$ such that
\begin{equation}\label{q:gronseq2}
       \x_n \le \x_{n-1} (1 + \dt \eta_{n-1}) + \dt \zeta_n,
          \qquad\textrm{for } n \geq n_0,
\end{equation}
and given the bounds
\begin{eqnarray}\label{q:groncond}
  \sum_{n=k_0}^{N+k_0} \dt \eta_n \le a_1, \qquad\qquad
  \sum_{n=k_0}^{N+k_0} \dt \zeta_n \le a_2, \qquad\qquad
  \sum_{n=k_0}^{N+k_0} \dt \x_n \le a_3,
\end{eqnarray}
for any $k_0 \geq n_0$, we have,
\begin{equation}\label{q:ugronest1}
       \x_n \le \Bigl( \frac{a_3}{N\dt}
         + a_2 \Bigr)\, {\rm e}^{a_1},  \qquad \forall n \geq N+n_0.
\end{equation}
\end{Lem}

We are now able to prove the following:

\begin{Prop}\label{t:bdh1}
Let $(u_0, \phi_0) \in \Vo$ and $(u^n, \phi^n)$ be the solution of
the numerical scheme (\ref{q-ies}). Also, let $\dt$ be such that
\begin{eqnarray}\label{t1}
      \dt  &\leq \min \left\{\frac{1}{\kappa }, 1 \right\}=:\kappa_1.
\end{eqnarray}
Then there exists $K_3\bigl( \|(u_{0},\phi_{0})\|_{\Vo},
\gvi\bigr)$, such that
\beq\label{2.108}
   \Arrowvert (u^n,\phi^n)\Arrowvert_{\Vo} \le K_3\bigl( \|(u_{0},\phi_{0})\|_{\Vo}, \gvi \bigr),
      \, \forall\, n \geq 0,
\eeq and for all $i=1, \cdots, m$, we have
\begin{eqnarray}\label{1.59}
\sum_{n=i}^{m} \left(\|{u}^n - {u}^{n-1}\|^2 + |A\g({\phi}^n -
{\phi}^{n-1})|^2_{L^2}\right) \leq
&K_3^2+c K_1^2 K_3^4 (m-i+1)\dt +\frac{2}{\nuo} \|g\|_{\infty}^2(m-i+1)\dt \no\\
& +  c \left(c_f^2 K_1^{2m} K_3^2 +  2( 2 c_f^2 +\alpha^{-2} \nu_2^2
\gamma^2)  K_1^2 \right) (m-i+1)\dt.
\end{eqnarray}

Moreover, there exists $K_4=K_4(\gvi)$,  independent of the
initial data, such that
\begin{equation}\label{1.86}
\Arrowvert (u^n,\phi^n)\Arrowvert_{\Vo} \leq K_4(\gvi), \, \forall n
\geq 2N_0+1,
\eeq
where $N_0:=\lfloor T_0/\dt\rfloor$, with $T_0$ being the time of entering an absorbing ball for 
$\|(u^n,\phi^n) \|_{\Yo}$.
\end{Prop}

\begin{proof}Using (\ref{1.71}), we infer from (\ref{1.83})
\begin{eqnarray}\label{2.83}
 \Arrowvert (u^n,\phi^n)\Arrowvert^2_{\Vo}  \leq  c_4 K_1^2 {\dt}\left[K_2 \|(u^{n-1},\phi^{n-1})\|^2_{\Vo}  + c_3(|f\g(\phi^n)|^2_{L^2}+|g^n|^2_{L^2})\right]^2
+\Arrowvert (u^{n-1},\phi^{n-1}) \Arrowvert^2_{\Vo} \no\\
+ c_5 \dt \left(| \fg'(\phi^n) \nabla \phi^n|^2_{L^2}
+|g^n|_{L^2}^2\right)
\Arrowvert (u^{n-1},\phi^{n-1}) \Arrowvert^2_{\Vo}\left[ 1+ 2c_4 K_1^2 K_2^2{\dt}\|(u^{n-1},\phi^{n-1})\|^2_{\Vo} \right] \no\\
 \quad + 2c_3^2 c_4 K_1^2 {\dt} (|f\g(\phi^n)|^2_{L^2}+|g^n|^2_{L^2})^2+c_5 \dt \left(| \fg'(\phi^n) \nabla \phi^n|^2_{L^2}+|g^n|_{L^2}^2\right),
\end{eqnarray}
which we rewrite in the form
\begin{equation}\label{2.107}
\x_n \le \x_{n-1} (1+\dt \eta_{n-1}) + \dt \zeta_n,
\end{equation}
 with
 \begin{eqnarray}
 \x_n=\Arrowvert (u^n,\phi^n)\Arrowvert^2_{\Vo}, \qquad \qquad \qquad  \eta_n=2c_4 K_1^2 K_2^2 \|(u^{n},\phi^{n})\|^2_{\Vo},\\
 \zeta_n=2c_3^2 c_4 K_1^2
(|f\g(\phi^n)|^2_{L^2}+|g^n|^2_{L^2})^2+c_5 \left(| \fg'(\phi^n)
\nabla \phi^n|^2_{L^2}+|g^n|_{L^2}^2\right). \no
 \end{eqnarray}
Using (\ref{fg}), (\ref{s1}), (\ref{s2}) and recalling
(\ref{q:bdv}) and the Sobolev imbedding $H^1(\Omega) \hookrightarrow
L^{q}(\Omega)$, for any $1\leq q \leq \infty$, we obtain
\begin{eqnarray}\label{2.84}
| f\g(\phi^n)|^2_{L^2}
\leq c c_f^2  \left(|\Omega|+K_1 ^{2(m+1)} \right)  + 2\alpha^{-2}\nu_2^2 \gamma^2 K_1^2,
\end{eqnarray}
\begin{eqnarray}\label{2.85}
| \fg'(\phi^n) \nabla \phi^n|^2_{L^2}
\leq c c_f^2 K_1 ^{2m} |A\g \phi^n|^2_{L^{2}} +  2( 2 c_f^2
+\alpha^{-2} \nu_2^2 \gamma^2)  K_1^2.
\end{eqnarray}
Relations (\ref{2.84}), (\ref{2.85}), together with  (\ref{M1}), (\ref{M2}),
and conclusion (\ref{q:gronest}) of
Lemma \ref{t:dgronwall} give
\begin{equation}\label{2.86}
   \x_n=\Arrowvert (u^n,\phi^n)\Arrowvert^2_{\Vo} \leq K_5^2(\|(u_{0},\phi_{0})\|_{\Vo},\gvi, 2N_0k
   \bigr), \quad \forall n=1, \cdots, 2N_0,
\end{equation}
for some continuous function $K_5(\cdot,\cdot, \cdot \bigr)$, increasing in all its arguments.

In order to derive an upper bound on
$\Arrowvert(u^n,\phi^n)\Arrowvert_{\Vo}$, $n \geq 2N_0$, we apply
Lemma \ref{t:dugronwall} to (\ref{2.107}). In order to do so, we
recall that $\|(u^n,\phi^n)\|_{\Yo}<\rho_0$, for $n\geq N_0$,
\begin{equation}\label{q:ugronest}
       \x_n=\Arrowvert (u^n,\phi^n)\Arrowvert^2_{\Vo} \le \Bigl( \frac{a_3}{T_0}
         + a_2 \Bigr)\, {\rm e}^{a_1}=:K_4^2(\gvi),  \qquad \forall n \geq
         2N_0+1,
\end{equation}
which is exactly (\ref{1.86}). Combining (\ref{q:ugronest}) with
(\ref{2.86}), we obtain conclusion (\ref{2.108}).

Taking the sum of (\ref{2.81}) with $n$ from $i$ to $m$ and using
(\ref{2.108}) and (\ref{2.85}) gives conclusion (\ref{1.59}) and
thus the proof of Proposition \ref{l3} is complete.
\end{proof}


\section{Convergence of Attractors}\label{s5}

In this section we address the issue of the convergence of the
attractors generated by the discrete system (\ref{q-ies}) to the
attractor generated by the continuous system (\ref{y6}). Whereas for
the continuous system (\ref{y6}) one can prove both the existence
and uniqueness of the solution (see \cite{gra1})---and, therefore,
define a global attractor---, for the discrete system (\ref{q-ies})
one can prove (using Proposition \ref{t:bdh}) the uniqueness of the
solution provided that $k\leq \kappa(\|(u_0, \phi_0)\|_{\Vo})$, for
some $ \kappa(\|(u_0, \phi_0)\|_{\Vo})>0$. Since the time
restriction depends on the initial data, one cannot define a
single-valued attractor in the classical sense, and this is why we
need to use the attractor theory for the so-called multi-valued
mappings. Multi-valued dynamical systems have been investigated by
many authors (see, e.g., \cite{Ball}, \cite{Barbu}, \cite{CLM},
\cite{MV}, \cite{RSS}, \cite{SSS}), but in this article we use the
tools developed in \cite{CZT} (see also, \cite{ET}) to study the
convergence of the discrete (multi-valued) attractors to the
continuous (single-valued) attractor.  For convenience, we recall
those results in Subsection \ref{ssec:abs}, and  then we apply them
to the two-phase flow model in Subsection \ref{ss:5.2}.


\subsection{Attractors for multi-valued mappings}\label{ssec:abs}
Throughout this subsection, we consider $(H, |\cdot|)$ to be a
Hilbert space and $\T$ to be either $\R^+=[0,\infty)$ or $\N$.

\begin{Def}
A one-parameter family of set-valued maps $S(t):2^H\to 2^H$ is a
\textbf{multi-valued semigroup} (m-semigroup) if it satisfies the
following properties:

\begin{enumerate}
    \item[(S.1)] $S(0)=I_{2^H}$ (identity in $2^H$); \label{S.1}
    \item[(S.2)] $S(t+s)=S(t)S(s)$, for all $t,s \in \T$. \label{S.2}
\end{enumerate}

Moreover, the m-semigroup is said to be \textbf{closed} if
      $S(t)$ is a closed map for every $t\in\T$, meaning that if $x_n\to x$ in $H$ and
    $y_n\in S(t)x_n$ is such that $y_n\to y$ in $H$, then $y\in S(t)x$.
\end{Def}


\begin{Def}
The \textbf{positive orbit} of $\B$, starting at $t\in\T$, is the
set
$$
\gamma_t(\B)=\bigcup_{\tau\geq t}S(\tau)\B,
$$
where
$$
S(t)\B=\bigcup_{x\in\B}S(t)x.
$$
\end{Def}


\begin{Def}
For any $\B\in2^H$, the set
$$
\omega(\B)=\bigcap_{t\in\T}\overline{\gamma_t(\B)}
$$
is called the \textbf{$\omega$-limit set} of $\B$.
\end{Def}


\begin{Def}
A nonempty set $\B\in 2^H$ is \textbf{invariant} for $S(t)$ if
$$
S(t)\B=\B, \qquad \forall t\in\T.
$$
\end{Def}


\begin{Def}
A set $\B_0\in 2^H$ is an \textbf{absorbing set} for the m-semigroup
$S(t)$ if for every bounded set $\B\in 2^H$ there exists $t_\B\in
\T$ such that
$$
S(t)\B\subset \B_0, \qquad \forall t\geq t_\B.
$$
\end{Def}


\begin{Def}
A nonempty set $\C\in 2^H$ is \textbf{attracting} if for every
bounded set $\B$ we have
$$
\lim_{t\to \infty}\dist(S(t)\B,\C)=0,
$$
where $\dist(\cdot,\cdot)$ is the \textbf{ Hausdorff semidistance},
defined as \beq\label{5.1a}
\dist(\B,\C)=\sup_{b\in\B}\inf_{c\in\C}|b-c|, \forall \, \B, \C
\subset  H. \eeq
\end{Def}


\begin{Def}
A nonempty compact set $\A\in 2^X$ is said to be the \textbf{global
attractor} of $S(t)$ if $\A$ is an invariant attracting set.
\end{Def}


\begin{Def}
Given a bounded set $\B\in 2^H$, the \textbf{Kuratowski measure of
noncompactness} $\alpha(\B)$ of $\B$ is defined as
$$
\alpha(\B)=\inf\big\{\delta\, :\, \B\ \textrm{has a finite cover by
balls of } X \textrm{ of diameter less than } \delta\big\}.
$$
\end{Def}



The following theorem, whose proof can be found in \cite{CZT}, gives
conditions under which a global attractor exists.
\begin{Thm}\label{teo:attr}
Suppose that the closed m-semigroup $S(t)$ possesses a bounded
absorbing set $\B_0\in 2^H$ and
\begin{equation}\label{eq:asymptcpt}
\lim_{t\to\infty}\alpha(S(t)\B_0)=0.
\end{equation}
Then $\omega(\B_0)$ is the global attractor of $S(t)$.
\end{Thm}


For the purpose of this article, we need to introduce the notion of
\textit{discrete m-semigroups}. More precisely, we have the
following:
\begin{Def}
Given a set-valued map $S:2^H\to 2^H$, we define a \textbf{discrete
m-semigroup} by
$$
S(n)=S^n, \qquad \forall n\in\N,
$$
and we will denote it by $\{S\}_{n\in\N}$ (instead of
$\{S^n\}_{n\in\N}$).
\end{Def}


\begin{Rem}
Given two nonempty sets $\B,\C\in 2^H$, we write
$$
\B-\C=\{b-c:\, b\in\B,\, c\in\C\} \qquad \textrm{and} \qquad
|\B|=\sup_{b\in\B} |b|.
$$
\end{Rem}


In order to prove the convergence of the attractors generated by the
discrete system (\ref{q-ies}) to the attractor generated by the
continuous system (\ref{y6}) we will use the following result, whose
proof can be found in \cite{CZT}  (see also \cite{XW},
\cite{TW2011}, \cite{ET}).
\begin{Thm}\label{teo:approx}
Let $S(t)$ be a closed m-semigroup, possessing the global attractor
$\A$, and for $\kappa_0>0$, let $\{S_k,\, 0<k\leq
\kappa_0\}_{n\in\N}$ be a family of discrete closed m-semigroups,
with global attractor $\A_k$. Assume the following:
\begin{enumerate}
    \item[(H1)]  [Uniform boundedness]:  there exists $\kappa_1\in (0,\kappa_0]$ such that the set
    $$
    \K=\bigcup_{k\in (0,\kappa_1]}\A_k
    $$
    is bounded in $H$;  \label{H1}
    \item[(H2)] [Finite time uniform convergence]: there exists $t_0\geq 0$ such that for any $T^\star>t_0$,
    $$
    \lim_{k\to 0}\sup_{x\in\A_k,\,nk\in [t_0,T^\star]}| S_k^nx-S(nk)x|=0.
    $$
    \label{H2}
\end{enumerate}
Then
$$
\lim_{k\to 0}\dist (\A_k,\A)=0,
$$
where $\dist$ denotes the Hausdorff semidistance defined in
(\ref{5.1a}).
\end{Thm}


\subsection{Application: The two-phase flow model}\label{ss:5.2}
Here we will prove that there exists
$\kappa_1>0$ such that if $0<k\leq \kappa_1$,  the system
(\ref{q-ies}) generates a closed discrete m-semigroup
$\{S_k\}_{n\in\N}$, with global attractors $\A_k$, that will
converge to $\A$ in the sense of Theorem \ref{teo:approx}.

In order to do that,  we define, for $\dt>0$, the multi-valued map
$S_k:2^\Yo\to 2^\Yo$ as follows: for every $\tilde{v}=(\tilde{u},
\tilde{\phi}) \in \Yo $,
$$
S_k \tilde{v}=\{v=(u,\phi) \in \Vo:\, v \textrm{ solves
(\ref{4.111}) below with time-step } k\}:
$$
\begin{equation}\label{4.111}
\left \{
\begin{array}{ll}
u + \nuo \dt A u + \dt B_0(u,u) - \K \dt R_0( \epso \Ag
\phi,\phi) = \tilde{u}+ \dt g,  \\
\mu = \epso \Ag \phi + \alpha \fg(\phi), \\
\phi + \dt \mu + \dt B_1(u, \phi) = \tilde{\phi}.
\end{array} \right.
\end{equation}


Using the same ideas as in \cite{CZT} (see also \cite{ET}), one can
prove the following:
\begin{Thm}
The multi-valued map $S_k$ associated with the implicit Euler scheme
(\ref{q-ies}) generates a closed discrete m-semigroup
$\{S_k\}_{n\in\N}$.
\end{Thm}


\begin{Prop}\label{prop:Vabs}
Let $k \leq \kappa_1$, where $\kappa_1$ is given in Proposition
\ref{t:bdh}. Then there exists a constant $R_1>0$ such that for
every $R\geq 0$ and $\|(u_0,\phi_0)\|_{\Yo}\leq R$, there exists
$N_1=N_1(R, \dt)\geq 0$ such that
\begin{equation}\label{eq:BBBB1}
\|S_k^n (u_0,\phi_0)\|_{\Vo}\leq R_1, \qquad \forall n\geq N_1.
\end{equation}
Thus, the set
$$
\B_1=\{(u,\phi)\in \Vo:\, \|(u,\phi)\|_{\Vo}\leq R_1\}
$$
is a $\Vo$-bounded absorbing set for $\{S_k\}_{n\in\N}$, for
$k\in(0,\kappa_1]$.
\end{Prop}


\begin{Prop}
For every $k\in (0,\kappa_1]$, there exists the global attractor
$\A_k$ of the m-semigroup $\{S_k\}_{n\in \N}$.
\end{Prop}


\begin{Rem}
Since the global attractor $\A_k$ is the smallest closed attracting
set of $\Yo$, Proposition \ref{prop:Vabs} also implies
\begin{equation}\label{5.21}
\A_k\subset \B_1, \forall k\in(0,\kappa_1],
\end{equation}
and thus
\begin{equation}\label{eq:unifbdd}
\bigcup_{k\in(0,\kappa_1]}\A_k \subset \B_1.
\end{equation}
\end{Rem}


From relation (\ref{eq:unifbdd}) we can see that condition ($H1$) of
Theorem \ref{teo:approx} is satisfied.  In order to prove
that condition ($H2$) is satisfied  we define, for
any function $\psi$ and for any $\dt>0$, the following:
\begin{equation}
\psi_k(t)=\psi^n, \quad t \in [(n-1)k, nk),
\end{equation}
\begin{equation}\label{tilde}
\tilde{\psi}_k(t)= \psi^n+\frac{t-nk}{k}(\psi^n-\psi^{n-1}), \quad
 t\in [(n-1)k, nk).
\end{equation}

With the above notations, the system (\ref{q-ies}) can be rewritten
as follows, for $t \in [(n-1)k, nk)$:
\begin{equation}\label{4:102}
 \left \{
\begin{array}{ll}
\Frac{d \ut}{d t} + \nuo A \ut + B_0(\ut,\ut) - \K R_0( \epso \Ag
\pt, \pt) = g+g_k(t), \\ \\
\Frac{d \pt}{d t} + \epso \Ag \phi_k + \alpha \fg(\phi_k) + B_1(\ut, \pt) = B_1(\ut,
\pt)-B_1(u_k(t), \phi_k(t)),
\end{array} \right. \end{equation}
where
\begin{eqnarray}\label{4:103}
g_k(t)
&= &\nuo A(\ut-u_k(t)) + B_0(\ut-u_k(t),\ut)+ B_0(u_k(t),\ut-u_k(t)) \no\\
&-& \K \left(R_0( \epso \Ag (\pt-\phi_k(t)), \pt)+\K R_0( \epso \Ag
\phi_k(t), \pt-\phi_k(t))\right).
\end{eqnarray}
Subtracting (\ref{4:102}) from (\ref{y6}) and setting
\begin{equation}
{{\xi}}_k(t)=u(t)-\tilde{u}_k(t), \qquad
\eta_k(t)=\phi(t)-\tilde{\phi}_k(t),
\end{equation}
we obtain
\begin{equation}\label{4:105}
\left \{
\begin{array}{ll}
\Frac{d \xi_k(t)}{d t} + \nuo A \xi_k(t) + B_0(\xi_k(t),u(t))+
B_0(\ut,\xi_k(t)) - \K \left(R_0( \epso \Ag \eta_k(t), \phi(t)) +
R_0(
\epso \Ag \pt, \eta_k(t))\right) = -g_k(t),  \\ \\
\Frac{d \eta_k(t)}{d t} + \epso \Ag \eta_k(t) + B_1(\xi_k(t),
\phi(t))+B_1(\ut, \eta_k(t)) = -\alpha \left(\fg(\phi(t))-
\fg(\pt)\right)-h_k(t),
\end{array} \right. \end{equation}
where $g_k(t)$ is given in (\ref{4:103}) and
\begin{eqnarray}\label{4:106}
h_k(t)= &B_1(\ut-u_k(t), \pt)+B_1(u_k(t), \pt-\phi_k(t))\\
&+\epso \Ag (\pt-\phi_k(t))+\alpha \left(\fg(\pt)-
\fg(\phi_k(t))\right). \no
\end{eqnarray}

\begin{Lem}\label{fk, gk}
Let $T^*>0$ be arbitrarily fixed and let $k< \kappa_1$, where
$\kappa_1$ is given in Proposition \ref{t:bdh}. Assume that $({u}_0,
\phi_0) \in \mathcal{A}_k$ and let $({u}^n, \phi^n)$ be the solution
of the numerical scheme (\ref{q-ies}). Then there exist $K_{6}(
T^*)$ and $K_{7}(T^*)$ such that
\begin{equation}\label{fk}
\|{g}_k\|_{L^2(0,T^*;{V}')}^2\leq \dt K_{6}(T^*),
\end{equation}
and
\begin{equation}\label{gk}
\|h_k\|_{L^2(0,T^*;D(A\g)')}^2\leq \dt K_{7}(T^*).
\end{equation}
\end{Lem}

\begin{proof}
We begin by noting that for any $t \in [(n-1)k, nk)$ we have
\begin{equation}\label{4.12}
\tilde{\psi}_k(t)-{\psi}_k(t)=\frac{t-nk}{k}(\psi^n-\psi^{n-1}).
\end{equation}
Also, since $(u_0,\phi_0) \in \mathcal{A}_k$, we have that
$\|(u_0,\phi_0)\|_{\Vo}\leq R_1$ (by (\ref{5.21})) and then by
Proposition \ref{t:bdh} we obtain that
\begin{equation}\label{5.33}
\|(u^n,\phi^n)\|_{\Vo} \leq K_3(R_1), \, \forall n \geq 0.
\end{equation}

Now let $v\in V$ be such that $\|v\|\leq 1$, and let $t\in
[(n-1)k,nk)$ be fixed. Using  (\ref{b0.1}) and
(\ref{b3.1}), and recalling (\ref{2.108}), we have the following bounds
\begin{equation}\label{4.133}
\quad |b_0(\ut-u_k(t),\ut,v)+ b_0(u_k(t),\ut-u_k(t),v)|
\leq c K_3\|u^n-u^{n-1}\|,
\end{equation}
\begin{eqnarray}\label{4.134}
\K \epso |b_1(v, \pt, \Ag (\pt-\phi_k(t)))+b_1(v, \pt-\phi_k(t),
\Ag\phi_k(t))|\\
\leq c \K \epso \left( |A\g \pt|_{L^2} + |\Ag \phi_k(t)|_{L^2}
\right)|\Ag (\pt-\phi_k(t))|_{L^2}
\leq c K_3 |\Ag (\phi^n-\phi^{n-1})|_{L^2}.\no
\end{eqnarray}
We also have
\begin{equation}\label{4.135}
\nuo |(A(\ut-u_k(t)), v)_{L^2}| \leq \nuo \|\ut-u_k(t)\| \|v\| \leq
\nuo \|u^n-u^{n-1}\|.
\end{equation}
Relations (\ref{4.133})--(\ref{4.135}) imply
\begin{equation}\label{4.136}
\|{g}_k(t)\|_{{V}'}\leq c K_3(\|u^n-u^{n-1}\|+|\Ag
(\phi^n-\phi^{n-1})|_{L^2}),
\end{equation}
and thus, setting $N^*=\lfloor T^\star/k \rfloor$ and recalling (\ref{4.136}) \textrm{ and
} (\ref{1.59}), we obtain
\begin{eqnarray}\label{4.137}
\|{g}_k\|_{L^2(0,T^*;{V}')}^2=\sum_{n=1}^{N^*+1} \int_{(n-1)\dt}
^{n\dt}\|{g}_k(t)\|_{{V}'}^2 dt \leq \dt K_6( T^*),
\end{eqnarray}
which proves (\ref{fk}).

Now let $\phi\in D(A\g)$ be such that $|A\g\phi|_{L^2}\leq 1$, and
let $t\in [(n-1)k,nk)$ be fixed. Using (\ref{b5.1}), (\ref{b3.1})
and (\ref{q:bdv}), we obtain
\begin{eqnarray}\label{4.138}
|b_1(\ut-u_k(t), \pt, \phi)| \leq c K_1 \|u^n-u^{n-1}\|,
\end{eqnarray}
\begin{eqnarray}\label{4.139}
|b_1(u_k(t), \pt-\phi_k(t), \phi)| \leq c K_3 |\phi^n-\phi^{n-1}|_{L^2}.
\end{eqnarray}
We also have
\begin{equation}\label{4.140}
|(\Ag (\pt-\phi_k(t)), \phi)_{L^2}|\leq |\pt-\phi_k(t)|_{L^2}
|A\g\phi|_{L^2} \leq |\phi^n-\phi^{n-1}|_{L^2}.
\end{equation}
\begin{eqnarray}\label{4.141}
|(\fg(\pt)- \fg(\phi_k(t)), \phi)_{L^2}|  \leq |\fg(\pt)-
\fg(\phi_k(t))|_{L^2} |\phi|_{L^2}\\
 \leq K_8 \|\pt-\phi_k(t)\|  \quad (\textrm{by } (\ref{fg}),  (\ref{s1}), (\ref{q:bdv}))
 \leq K_8 \|\phi^n-\phi^{n-1}\|. \no
\end{eqnarray}
Gathering relations (\ref{4.138})--(\ref{4.141}) we obtain
\begin{equation}\label{4.142}
\|h_k(t)\|_{D(A\g)'}\leq c(K_1+K_3+K_8)(\|u^n-u^{n-1}\|+|\Ag
(\phi^n-\phi^{n-1})|_{L^2}),
\end{equation}
and recalling (\ref{4.142}) \textrm{ and } (\ref{1.59})
we obtain (\ref{gk}). This completes the proof of the lemma.
\end{proof}


We are now in a position to prove that condition (H2) of Theorem
\ref{teo:approx} is satisfied.

\begin{Prop}[Finite time uniform convergence]
For any $T^*>0$ we have
\begin{equation}
\lim_{k \to 0} \sup_{(u_0,\phi_0) \in \mathcal{A}_k, \,nk \in
[0,T^*]} \|S_k^n (u_0,\phi_0)-S(nk)(u_0,\phi_0)\|_{\Yo}=0.
\end{equation}
\end{Prop}

\begin{proof}
Multiplying the first equation of (\ref{4:105}) by $\xi(t)$ and
integrating we obtain
\begin{eqnarray}\label{4.145}
\frac{1}{2} \frac{d}{dt} |\xi_k(t)|^2_{L^2}+\nu_1
\|\xi_k(t)\|^2+b_0(\xi_k(t), u(t), \xi_k(t))\\
-\K \left(b_1( \xi_k(t),  \phi(t), \epso \Ag \eta_k(t))
+b_1(\xi_k(t), \eta_k(t), \epso \Ag \pt)\right) = -(g_k(t),\no
\xi_k(t))_{L^2}.
\end{eqnarray}
Using (\ref{b0.2}) and (\ref{b3.1}) we bound the nonlinear terms as
follows:
\begin{eqnarray}\label{4.146}
|b_0(\xi_k(t), u(t), \xi_k(t))| \leq c_b |\xi_k(t)|_{L^2}
\|\xi_k(t)\| \|u(t)\| \\
\leq \frac{\nuo}{8} \|\xi_k(t)\|^2 + c |\xi_k(t)|_{L^2}^2
\|u(t)\|^2, \no
\end{eqnarray}
\begin{eqnarray}\label{4.147}
\K |b_1( \xi_k(t),  \phi(t), \epso \Ag \eta_k(t))| \leq c_b \K
\epso |\xi_k(t)|^{1/2}_{L^2} \|\xi_k(t)\|^{1/2} \|\phi(t)\|^{1/2}
|A\g \phi(t)|^{1/2}_{L^2} |\Ag \eta_k(t)|_{L^2}\no\\
 \leq \frac{\nu_2^2}{4}\K |\Ag \eta_k(t)|_{L^2}^2+ \frac{\nu_1}{8}
\|\xi_k(t)\|^{2}  + c |\xi_k(t)|^{2}_{L^2} \|\phi(t)\|^{2} |A\g
\phi(t)|^{2}_{L^2},
\end{eqnarray}
\begin{eqnarray}\label{4.148}
\K|b_1(\xi_k(t), \eta_k(t), \epso \Ag \pt) \leq c_b \K \epso
|\xi_k(t)|^{1/2}_{L^2} \|\xi_k(t)\|^{1/2} \|\eta_k(t)\|^{1/2} |A\g
\eta_k(t)|^{1/2}_{L^2} |\Ag \pt|_{L^2}\no\\
 \leq \frac{\nu_2^2}{4} \K |\Ag \eta_k(t)|_{L^2}^2+ \frac{\nu_1}{8}
\|\xi_k(t)\|^{2}  + c |\xi_k(t)|_{L^2} \|\eta_k(t)\| |\Ag
\pt|^{2}_{L^2}.
\end{eqnarray}
Using the Cauchy-Schwarz inequality, we bound the right-hand side of
(\ref{4.145}) as
\begin{equation}\label{4.149}
|(g_k(t), \xi_k(t))_{L^2}| \leq \|g_k(t)\|_{{V}'} \|\xi_k(t)\| \leq
\frac{\nu_1}{8} \|\xi_k(t)\|^{2} + c \|g_k(t)\|_{{V}'}^2.
\end{equation}
Relations (\ref{4.145})--(\ref{4.149}) imply
\begin{eqnarray}\label{4.150}
 \frac{d}{dt} |\xi_k(t)|^2_{L^2}+\nu_1
\|\xi_k(t)\|^2\leq c |\xi_k(t)|_{L^2}^2 \|u(t)\|^2 +
\frac{\nu_2^2}{2}\K |\Ag \eta_k(t)|_{L^2}^2+ c |\xi_k(t)|^{2}_{L^2}
\|\phi(t)\|^{2} |A\g \phi(t)|^{2}_{L^2} \no\\
 + c \left(|\xi_k(t)|_{L^2}^2+ \|\eta_k(t)\|^2\right) |\Ag
\pt|^{2}_{L^2} + c \|g_k(t)\|_{{V}'}^2.
\end{eqnarray}

Now multiplying the second equation of (\ref{4:105}) by $\epso
A\g\eta_k(t)$ and integrating we obtain
\begin{eqnarray}\label{4.151}
 \frac{\epso}{2} \frac{d}{dt} \|\eta_k(t)\|^2_{\gamma} + \nu_2^2 |\Ag \eta_k(t)|_{L^2}^2 +  b_1(\xi_k(t),
\phi(t), \epso A\g\eta_k(t))+ b_1(\ut, \eta_k(t), \epso A\g\eta_k(t)) \no\\
= -\alpha (\fg(\phi(t))- \fg(\pt), \epso A\g\eta_k(t)
)_{L^2}-(h_k(t), \epso A\g\eta_k(t))_{L^2}.
\end{eqnarray}
Using (\ref{b3.1}) we bound the nonlinear terms as follows:
\begin{eqnarray}\label{4.152}
|b_1(\xi_k(t), \phi(t), \epso A\g\eta_k(t))| \leq c_b  \epso
|\xi_k(t)|^{1/2}_{L^2} \|\xi_k(t)\|^{1/2} \|\phi(t)\|^{1/2} |A\g
\phi(t)|^{1/2}_{L^2} |\Ag \eta_k(t)|_{L^2}\no\\
 \leq \frac{\nu_2^2}{8} |\Ag \eta_k(t)|_{L^2}^2+ \frac{\nu_1}{4\K}
\|\xi_k(t)\|^{2}  + c |\xi_k(t)|^{2}_{L^2} \|\phi(t)\|^{2} |A\g
\phi(t)|^{2}_{L^2},
\end{eqnarray}
\begin{eqnarray}\label{4.153}
|b_1(\ut, \eta_k(t), \epso A\g\eta_k(t))| \leq c_b  \epso
|\ut|^{1/2}_{L^2} \|\ut\|^{1/2} \|\eta_k(t)\|^{1/2} |A\g
\eta_k(t)|^{3/2}_{L^2} \no\\
\leq \frac{\nu_2^2}{8} |\Ag \eta_k(t)|_{L^2}^2+ c|\ut|^{2}_{L^2}
\|\ut\|^{2} \|\eta_k(t)\|^{2}.
\end{eqnarray}
Using the Cauchy-Schwarz inequality, we bound the right-hand side of
(\ref{4.151}) as
\begin{eqnarray}\label{4.154}
\qquad \qquad \alpha |(\fg(\phi(t))- \fg(\pt), \epso A\g\eta_k(t))_{L^2}|\leq
\epso |\fg(\phi(t))- \fg(\pt)|_{L^2} |A\g\eta_k(t)|_{L^2}\\
\leq \frac{\nu_2^2}{8} |\Ag \eta_k(t)|_{L^2}^2+c |\fg(\phi(t))-
\fg(\pt)|_{L^2}^2 \leq \frac{\nu_2^2}{8} |\Ag \eta_k(t)|_{L^2}^2+c K_8^2
\|\eta_k(t)\|^2 \no
\end{eqnarray}
\begin{eqnarray}\label{4.155}
|(h_k(t), \epso A\g\eta_k(t))_{L^2}| &\leq& \epso \|h_k(t)\|_{D(A\g)'}
|A\g\eta_k(t))|_{L^2}\no\\
&\leq& \frac{\nu_2^2}{8} |\Ag \eta_k(t)|_{L^2}^2+ c \|h_k(t)\|_{D(A\g)'}^2.
\end{eqnarray}

Relations (\ref{4.151})--(\ref{4.155}) yield
\begin{eqnarray}\label{4.156}
{\epso} \frac{d}{dt} \|\eta_k(t)\|^2_{\gamma} + \nu_2^2 |\Ag
\eta_k(t)|_{L^2}^2 \leq \frac{\nu_1}{2\K} \|\xi_k(t)\|^{2}  + c
|\xi_k(t)|^{2}_{L^2} \|\phi(t)\|^{2} |A\g \phi(t)|^{2}_{L^2}\no\\
+c|\ut|^{2}_{L^2} \|\ut\|^{2} \|\eta_k(t)\|^{2}+c K_7^2
\|\eta_k(t)\|^2+ c \|h_k(t)\|_{D(A\g)'}^2.
\end{eqnarray}

Dividing (\ref{4.150}) by $\K$ and adding the resulting equation to
(\ref{4.156}) we obtain
\begin{eqnarray}\label{4.157}
 \frac{d}{dt} \|(\xi_k(t), \eta_k(t)\|^2_{\Yo}+\frac{\nu_1}{2\K}
\|\xi_k(t)\|^2 + \frac{\nu_2^2}{2} |\Ag \eta_k(t)|_{L^2}^2 \leq c
|\xi_k(t)|_{L^2}^2 \|u(t)\|^2 + c |\xi_k(t)|^{2}_{L^2}
\|\phi(t)\|^{2} |A\g \phi(t)|^{2}_{L^2} \no\\
 +  c \left(|\xi_k(t)|_{L^2}^2+ \|\eta_k(t)\|^2\right) |\Ag
\pt|^{2}_{L^2}
+ c \|g_k(t)\|_{{V}'}^2+c|\ut|^{2}_{L^2} \|\ut\|^{2}
\|\eta_k(t)\|^{2}
+c K_8^2 \|\eta_k(t)\|^2+ c \|h_k(t)\|_{D(A\g)'}^2.\no
\end{eqnarray}
Neglecting some positive terms, the above relation implies
\begin{eqnarray}\label{4.158}
 \frac{d}{dt} \|(\xi_k(t), \eta_k(t)\|^2_{\Yo} \leq  \G(t) \|(\xi_k(t),
 \eta_k(t)\|^2_{\Yo}+c \|g_k(t)\|_{{V}'}^2+ c \|h_k(t)\|_{D(A\g)'}^2,
\end{eqnarray}
where
\begin{eqnarray}\label{4.159}
 \G(t)=c\left( \|u(t)\|^2+\|\phi(t)\|^{2} |A\g \phi(t)|^{2}_{L^2} +|\Ag
\pt|^{2}_{L^2}+|\ut|^{2}_{L^2} \|\ut\|^{2}+K_8^2\right). \no
\end{eqnarray}
By the Gronwall Lemma and using the fact that $\xi_k(0)=
\eta_k(0)=0$, we obtain
\begin{equation}\label{4.160}
\quad \|(\xi_k(t), \eta_k(t)\|^2_{\Yo} \leq  c\int_0^t
\exp{\left(\int_\tau^t \G(s) \, ds\right)}
\left(\|g_k(\tau)\|_{{V}'}^2+ \|h_k(\tau)\|_{D(A\g)'}^2\right) \, d
\tau.
\end{equation}
Using the fact that the solution $({u}, \phi)$ of the continuous problem is
uniformly bounded in $\Vo$ for all $t\ge0$ (cf.\ \cite{gra1}),
and recalling (\ref{tilde}) and  (\ref{2.108}), we obtain
\begin{equation}\label{4.164}
\int_\tau^t \G(s) \, ds \leq c_6,
\end{equation}
for some constant  $c_6=c_6(T^*)>0$.

Relations (\ref{4.160}), (\ref{4.164}), (\ref{fk}) and (\ref{gk})
give
\begin{equation}\label{4.165}
\|(\xi_k(t), \eta_k(t)\|^2_{\Yo} \leq  \dt c_7,
\end{equation}
and thus
\begin{eqnarray}\label{4.166}
&&\lim_{k \to 0} \sup_{(u_0,\phi_0)\in \mathcal{A}_k, \,nk \in
[0,T^*]} \|S_k^n (u_0,\phi_0)-S(nk)(u_0,\phi_0)\|_{\Yo}\\
&=&\lim_{k \to 0} \sup_{(u_0,\phi_0) \in \mathcal{A}_k, \,nk \in
[0,T^*]} \sup_{(u^n, \phi^n)\in S_k^n (u_0,\phi_0)} \| (u^n,\phi^n)-(u(nk),\phi(nk))\|_{\Yo}\no\\
&=&\lim_{k \to 0} \sup_{(u_0,\phi_0) \in \mathcal{A}_k, \,nk \in
[0,T^*]}\sup_{(u^n, \phi^n)\in S_k^n (u_0,\phi_0)} \|(\tilde{u}_k(nk), \tilde{\phi}_k(nk))-(u(nk),\phi(nk))\|_{\Yo}\no\\
&=&\lim_{k \to 0} \sup_{(u_0,\phi_0) \in \mathcal{A}_k, \,nk \in
[0,T^*]} \sup_{(u^n, \phi^n)\in S_k^n (u_0,\phi_0)}\|(\xi_k(nk),
\eta_k(nk))\|_{\Yo}=0,\no
\end{eqnarray}
which concludes the proof of the lemma.
\end{proof}


Having proved that conditions (H1) and (H2) of Theorem
\ref{teo:approx} are satisfied we also obtain that the discrete
attractors converge to the continuous attractor as the time-step
approaches zero. More precisely, we have the following:

\begin{Thm}\label{teo:apprrrr}
The family of attractors $\{\A_k\}_{k\in(0,\kappa_1]}$  converges,
as $k\to 0$, to  $\A$, in the following sense:
$$
\lim_{k\to 0}\dist(\A_k,\A)=0,
$$
where $\dist$ denotes the Hausdorff semidistance in $\Yo$, namely
$$
\dist(\A_k,\A)=\sup_{x_k\in\A_k}\inf_{x\in\A}\|x_k-x\|_{\Yo}.
$$
\end{Thm}


\bibliography{biblio}

\begin{thebibliography}{10}

\bibitem{Ball}{\sc J.M. Ball},
{\em Continuity properties and global attractors of generalized semiflows
  and the {N}avier--{S}tokes equations}, J.\ Nonlinear Sci., 7 (1997), pp.~475--502.

\bibitem{Barbu}{\sc V.~Barbu and S.S. Sritharan},
{\em$ {H}^{\infty} $ control theory in fluid dynamic},  Proc. R. Soc. London A, 454 (1998), pp.~3009--3033.

\bibitem{bles} {\sc T.~Blesgen},
{\em A generalization of the {N}avier-{S}tokes equation to two-phase flow}, Pysica D (Applied Physics), 32 (1999), pp.~1119--1123.

\bibitem{cagi}{\sc G.~Caginalp},
{\em An analysis of a phase field model of a free boundary}, Arch. Rational Mech. Anal., 92(3) (1986), pp.~205--245.

\bibitem{CLM}{\sc T.~Caraballo, J.A. Langa, V.S. Melnik, and J.~Valero},
{\em Pullback attractors of nonautonomous and stochastic multivalued
  dynamical systems}, Set-Valued Anal., 2 (2003), pp.~153--201.

\bibitem{CZT}{\sc M.~Coti-Zelati and F.~Tone},
{\em Multivalued attractors and their approximation: Applications to the
  {N}avier--{S}tokes equations},  Numerische Mathematik, 122 (2012), pp 421--441.

\bibitem{ET}{\sc B.~Ewald and F.~Tone}
{\em Approximation of the long-term dynamics of the dynamical system
  generated by the two-dimensional thermohydraulics equations},International Journal of Numerical Analysis and Modeling,
  10(3) (2013), pp.~509--535.

\bibitem{fei1}{\sc E.~Feireisl, H.~Petzeltov\'a, E.~Rocca, and G.~Schimperna},
{\em Analysis of a phase-field model for two-phase compressible fluids}, Math. Models Methods Appl. Sci., 20(7) (2010), pp.~1129–1160.

\bibitem{gra2}{\sc C.~G. Gal and M.~Grasselli},
{\em Asymptotic behavior of a {C}ahn-{H}illiard-{N}avier-{S}tokes system
  in 2{D}}, Ann. Inst. H. Poincar\'{e} Anal. Non Lin\'{e}aire, 27(1) (2010), pp.~401–436.

\bibitem{gra1}{\sc C.~G. Gal and M.~Grasselli},
{\em Longtime behavior for a model of homogeneous incompressible two-phase
  flows}, Discrete Contin. Dyn. Syst., 28(1) (2010), pp.~1--39.

\bibitem{gra3}{\sc C.~G. Gal and M.~Grasselli},
{\em Trajectory attractors for binary fluid mixtures in 3{D}}, Chin. Ann. Math. Ser. B, 31(5) (2010), pp.~655–678.

\bibitem{MV}{\sc V.S. Melnik and J.~Valero},
{\em On attractors of multivalued semi-flows and differential inclusion}, Set-Valued Anal., 6 (1998), pp.~83--111.

\bibitem{RSS}{\sc R.~Rossi, S.~Segatti, and U.~Stefanelli},
{\em Attractors for gradient flows of nonconvex functionals and
  applications}, Arch.\ Ration.\ Mech.\ Anal., 187 (2008), pp.~91--135.

\bibitem{SSS}{\sc G.~Schimperna, S.~Segatti, and U.~Stefanelli}
{\em Well-posedness and long-time behavior for a class of doubly nonlinear
  equations},  Discrete Contin.\ Dyn.\ Syst., 18 (2007), pp.~15--38.

\bibitem{shen:90}{\sc J. Shen}
{\em Long time stabilities and convergences for the fully discrete
  nonlinear {G}alerkin methods},  Appl. Anal., 38(4) (1990), pp.~201--229.

\bibitem{temam88}{\sc R.~Temam}
 {\em Infinite dimensional dynamical systems in mechanics and
  physics}, volume~68, Appl. Math. Sci., Springer-Verlag, New York, second edition, 1988.

\bibitem{T4}{\sc F. Tone},
{\em On the long-time ${H}^2$-stability of the implicit {E}uler scheme for
  the 2{D} magnetohydrodynamics equations}, Journal of Scientific Computing, 38 (2009), pp.~331--348.

\bibitem{TW2011}{\sc F. Tone and X. Wang},
{\em Approximation of the stationary statistical properties of the
  dynamical system generated by the two-dimensional {R}ayleigh-{B}enard
  convection problem}, Analysis and Applications, 09(4) (2011), pp.~421--446.

\bibitem{TW}{\sc F. Tone and D. Wirosoetisno},
{\em On the long-time stability of the implicit {E}uler scheme for the
  2{D} {N}avier--{S}tokes equations}, SIAM Journal on Numerical Analysis, 44(1) (2006), pp.~29--40.

\bibitem{XW}{\sc X. Wang},
{\em Approximation of stationary statistical properties of dissipative
  dynamical systems: time discretization}, Math. Comp., 79(269) (2010), pp.~259--280.

\end{thebibliography}

\end{document}